\documentclass[12pt]{amsart}
\usepackage{amssymb}
\usepackage{mathtools}
\usepackage{mathrsfs}
\usepackage{amsthm}
\usepackage{tikz}
\usepackage{enumitem}
\usepackage[margin=1in]{geometry}

\DeclareMathOperator{\Aut}{Aut}	
\DeclareMathOperator*{\dlim}{\underrightarrow{\lim}}
\DeclareMathOperator{\End}{End}	
\DeclareMathOperator{\Hom}{Hom}	
	
\DeclareMathOperator{\id}{id}	
\DeclareMathOperator*{\ilim}{\underleftarrow{\lim}}	

\DeclareMathOperator{\Perf}{Perf}

\DeclareMathOperator{\Spa}{Spa}

\DeclareMathOperator{\Sym}{Sym}	

\newcommand{\LS}{\underline{\mathbb B^+_\dR\operatorname{-LocSys}}}
\newcommand{\ls}{\mathbb B^+_\dR\operatorname{-LocSys}}
\newcommand{\tMIC}{\underline{\operatorname{\text{$t$}-MIC}}}
\newcommand{\tmic}{\operatorname{\text{$t$}-MIC}}

\usepackage{amsmath, amsthm, amssymb, appendix, bm, graphicx, hyperref, mathrsfs, tikz-cd,tikz,indentfirst}
\usepackage{url, amssymb,setspace,fontenc, comment,stmaryrd}
\usepackage{fullpage} 
\usepackage{verbatim}
\usepackage{exscale}
\usepackage[all,ps,tips,tpic]{xy}

\DeclareMathOperator{\Fil}{Fil}
\DeclareMathOperator{\gr}{gr}

\def\id{\mathrm{id}}

\newcommand{\OO}{\mathcal{O}}

\newcommand{\BB}{\mathbb{B}}
\DeclareMathOperator{\A}{A}
\DeclareMathOperator{\B}{B}
\newcommand{\bdr}{{\BB_\dR^+}}
\newcommand{\obdr}{\OO\BB_\dR}

\def \dR {\text{\normalfont{dR}}}
\def \d {\text{\normalfont{d}}}

\def \et {\text{\normalfont{\'et}}}
\def \Fil {\text{\normalfont{Fil}}}	
\def \gr {\text{\normalfont{gr}}}	
\def \proet {\text{\normalfont{pro\'et}}}
\def \sm {\text{\normalfont{sm}}}	
\def \adj {\text{\normalfont{adj}}}

\title{A stacky nilpotent $p$-adic Riemann-Hilbert correspondence}
\author{Yudong Liu, Chenglong Ma, Xiecheng Nie, Xiaoyu Qu}
\date{}
\begin{document}
\maketitle
\begin{abstract}
 Let $\overline X$ be a smooth rigid variety over $C=\mathbb C_p$ admitting a lift $X$ over $B_\dR^+$. In this paper, we use the stacky language to prove a nilpotent $p$-adic Riemann-Hilbert correspondence. After introducing the moduli stack of $\mathbb B^+_\dR$-local systems and $t$-connections, we prove that there is an equivalence of the nilpotent locus of the two stacks: $RH^0:\LS^0_X\to \tMIC^0_X$, where $\LS^0_X$ is the stack of nilpotent $\mathbb B^+_\dR$-local systems on $\overline X_{1,v}$ and $\tMIC^0_X$ is the stack of $\OO_X$-bundles with integrable $t$-connection on $X_\et$. 
\end{abstract}
\setcounter{tocdepth}{1}
\tableofcontents
\newtheorem{thm}{Theorem}[section]
\newtheorem{lem}[thm]{Lemma}
\newtheorem{prop}[thm]{Proposition}
\newtheorem{cor}[thm]{Corollary}
\theoremstyle{definition}\newtheorem{defn}[thm]{Definition}
\theoremstyle{definition}\newtheorem{defns}[thm]{Definitions}
\theoremstyle{remark}\newtheorem*{rmk}{Remark}

\numberwithin{equation}{section}

\section{Introduction}
\subsection{Overview}
The classical Riemann-Hilbert correspondence states that for a smooth projective variety $X/\mathbb C$, there exists a categorical equivalence 
\begin{align*}
\{\mathbb{C}\text{-local systems on } X_{\et}\} &\longleftrightarrow \{\text{flat connections on }X\}\\
\mathbb{L} &\mapsto (\mathbb{L} \otimes_{\mathbb{C} }\mathcal{O} _X, \id_{\mathbb{L} }\otimes d)\\
D^{\nabla=0} &\mapsfrom (D,\nabla)
\end{align*}

A $p$-adic analogue of such a correspondence is sought-after for arithmetic purposes. In this setting, one considers e.g. a smooth adic space $X/\Spa(k,\mathcal O_k)$ for a discrete complete nonarchimedean extension $k/\mathbb Q_p$ with perfect residue.

Scholze makes use of filtered modules with integrable connection to develop his $p$-adic Hodge theory \cite{S2} and shows that they correspond to deRham
$\mathbb B^+_\dR$-local systems. Liu and Zhu in \cite{LZ} improve his method and obtain a tensor functor
\begin{align*}
\{\mathbb Q_p\text{-local systems on }X_{\et}\} \to{}& \{\text{filtered Galois $\mathcal O_X\widehat\otimes B_\dR$-modules with}\\ 
 &\text{integral connection on }X\}\\
\mathbb{L} \mapsto{}& (\nu_*(\nu^*\mathbb{L} \otimes \mathcal O \mathbb B_\dR), \id_{\mathbb{L} }\otimes d)
\end{align*}

Note that in those cases the Higgs field $\gr^0(\nabla)$ is automatically nilpotent. 

Recently, on a smooth adic space $X$ over $B^+_\dR$, Yu shows that the $\mathbb B^+_\dR$-local systems and $t$-connections both forms small v-stacks in the sense of \cite{S3}, and proves a sheafified correspondence for them in \cite{Yu}.
We then aim to show a stacky nilpotent correspondence in his setting.

\subsection{Main Results}
Let $\overline X$ be a smooth rigid variety over $C=\mathbb C_p$ admitting a lift $X$ over $B_\dR^+$. We consider the correspondence between 
$\BB_{\dR,\overline X}^+$-local systems with nilpotent Sen operator, and integrable t-connections with nilpotent Higgs field on $X$. We will call the objects of both sides nilpotent for short.
Our main result for this problem is
\begin{thm}\label{main}
There is an equivalence of stacks $RH^0:\LS^0_X\to \tMIC^0_X$.
\end{thm}

Note that our work may easily adapt to the arithmetic case. We may regard it as a generalization of \cite{LZ}. It also answers the question \cite[Remark 3.1]{LZ} asking for the geometric sense of $\mathcal X^+$ in their words.

This work is based on Yu's detailed study of adic spaces over $B^+_\dR$. And our approach is similar to \cite{LZ}.

\begin{cor}\label{cor:whole correspondence for P}
For the projective adic space $\mathbb P^d$ over $\mathrm B_\dR^+$ there is an equivalence of the whole moduli stacks
\[\LS_{\mathbb P^d}\simeq \tMIC_{\mathbb P^d}.\]
\end{cor}
\begin{proof}
For $P$ the Hitchin base $A_P=0$.
\end{proof}

\begin{cor}
    
\end{cor}

\begin{rmk}
Further, in a forthcoming work \cite{lmnqw} we construct period sheaves with more structures and generalize our result to small objects:

There is an equivalence of stacks of Hitchin-small objects
\[RH^\sm:\LS^\sm_X\to \tMIC^\sm_X.\]

By mod t we refind Anschütz, Heuer and Le Bras's small Simpson correspondence \cite{h}. Our method provides a direct expression of their correspondence.    
\end{rmk}

\subsection{Relation to other works}
There are traditionally two ways to produce a $p$-adic Riemann-Hilbert correspondence. The first one is the approach by Liu-Zhu \cite{LZ}, by constructing the 'correct' version of period sheaves in various situations and writing down two sides of the functors explicitly and directly proving equivalences. The other approach is given by work of Heuer, and Anschütz-Heuer-Le Bras \cite{h}, who used stacky approach and apply constructions of the Hodge-Tate stack.

Our method can be seen as a mix of both. We use the construction of moduli stacks, first introduced by Heuer, later refined by Yu who proved his version of moduli stacks are small v-stacks in the sense of \cite{S3}. We also construct a global period sheaf and write down explicit formulas for the functors.

\subsection{Acknowledgement}

This work is a research project of 2024 Algebra and Number Theory Summer School at Peking University, held by School of Mathematical Science, Peking University and Academy of Mathematics and System Science, Chinese Academy of Science. The authors would like to thank Yupeng Wang for introducing this problem and helpful discussions. Moreover, we would like to thank Yupeng Wang for carefully reading the preliminary draft of this paper. The authors would like to thank the organizers and institutes for providing a great collaborative environment. We would like to thank Jiahong Yu for explaining his paper \cite{Yu} and thank Yue Chen and Wenrong Zou for their interests in this project.

\subsection{Notations}

Throughout this paper, let $C$ be an algebraically closed perfectoid field containing $\mathbb Q_p$ with the ring of integers $\OO_C$. Let $\Perf_C$ denote the big v-site of perfectoid over $(C,\OO_C)$. Let $\A_{\inf}$ and $\B_\dR^+$ be its infinitesimal and de Rham period ring. Fix an embedding $p^{\mathbb Q}\subset C^{\times}$, which induces an embedding $\varpi^{\mathbb Q}\subset C^{\flat\times}$, where $\varpi = (p,p^{1/p},p^{1/p^2},\dots)\in C^{\flat}$. Fix a coherent system $\{\zeta_{p^n}\}_{n\geq 0}$ of primitive $p^n$-th roots of unity in $C$, and let $\epsilon:=(1,\zeta_p,\zeta_p^2,\dots)\in C^{\flat}$. Let $u = [\epsilon^{\frac{1}{p}}]-1\in\A_{\inf}$, and then the canonical surjection $\theta:\A_{\inf}\to \OO_C$ is principally generated by $\xi:=\frac{\phi(u)}{u}$. Let $t = \log([\epsilon])\in \B_\dR^+$ be the Fontaine's $p$-adic analogue of ``$2\pi i$''. For any (sheaf of) $\A_{\inf}$-module $M$ and any $n\in \mathbb Z$, denote by $M\{n\}$ its Breuil--Kisin--Fargues twist 
  \[M\{n\}:=M\otimes_{\A_{\inf}}\ker(\theta)^{\otimes n},\]
  which can be trivialized by the identification $M\{n\} = M\cdot\xi^n$.

Galois cohomology always means continuous Galois cohomology.

For a presheaf $F$ on some site denote $F^\#$ to be its sheafification.

For a $\mathbb Z$-filtered object $\Fil^\bullet F$ in some abelian category,  denote $F^{[a,b]}=\Fil^aF/\Fil^{b+1}F$ for any $-\infty \leq a\leq b \leq \infty$, where $\Fil^\infty F=0$ and $\Fil^{-\infty} F=F$.

An \'etale morphism between affinoid spaces is called standard \'etale if it is a composition of rational localizations and finite \'etale morphisms.

\section{Preliminaries}
We lay our foundations on \cite{Yu}. In this section we recall some definitions in \cite{Yu} and collect some basic properties.

\begin{defn}
By a smooth adic space $X$ over $B_\dR^+$, we mean 
a system of adic spaces $\bigl\{ X_\alpha \bigr\}_{\alpha \in \mathbb Z_{\geq 1}}$ such that for each $\alpha\geq 1$, the $X_\alpha$ is smooth over $\Spa(B_{\dR,\alpha}^+,B_{\dR,\alpha}^{++})$, and that for $\alpha'>\alpha$, it holds that
\[ X_{\alpha'} = X_\alpha \times_{B_{\dR,\alpha}^+} B_{\dR,\alpha'}^+. \]

We will always denote $X_1$ by $\overline X$. And if $X=\Spa\bigl(A,A^+\bigr)$ is also affinoid, then $A_\alpha=A/t^\alpha$.
\end{defn}

\begin{prop}\label{etsite}
Let $X$ be a smooth adic space over $B_\dR^+$, then for $\alpha < \alpha',$ the base change 
\[ X_{\alpha',?} \to X_{\alpha,?},\quad ?\in\{\et,\proet\} \]
are equivalences of sites.
\end{prop}
\begin{proof}
Clear by formal smoothness
\end{proof}

Denote by $X_?$ the sites
\[ X_? = \ilim X_{\alpha,?},\quad ?\in\{\et,\proet\} \]
and let $\nu: X_{1,v} \to X_\et$ be the natural projection.

\begin{defn}
Let $X$ be a smoothoid adic space over $\text{B}_\dR^+$. Define the categories
\begin{align*}
\ls(X)&=\{\text{$\mathbb B^+_\dR$-local systems on $X_{1,v}$}\},\\
\tmic(X)&=\{\text{$\OO_X$-bundles with integrable t-connection on $X_\et$}\}
\end{align*}
\end{defn}

\subsection{Moduli stacks}

Let $X$ be a smooth adic space over $\text{B}_\dR^+$. For any $S\in\Perf_C$, denote $X_S=X\times_{\text{B}^+_\dR}B^+_\dR(S)$.

In \cite{Yu}, Yu considered the following two functors from $\Perf_C$ to the category of groupoids:
\begin{align*}
\LS_X:&\ S\in \Perf_C\mapsto \ls(X_S),\\
\tMIC_X:&\ S\in \Perf_C\mapsto \tmic(X_S)
\end{align*}
For any morphism $f:S'\to S$ in $\Perf_C$, 
we have \begin{align*}
f^*_v\mathbb L=&f^{-1}_v\mathbb L\otimes_{f_v^{-1}\mathbb B_\dR^+}\mathbb B_\dR^+,\\
f^*_\et\mathcal E=&f^{-1}_\et\mathcal E\otimes_{f_\et^{-1}\mathcal O_{X_S}}\mathcal O_{X_{S'}}.
\end{align*}
Moreover, $(f^*_v,f_{*,v})$, $(f^*_\et,f_{*,\et})$ are adjoint functors between categories of sheaves of modules.

\begin{thm}\cite[Theorem 4.14]{Yu}
  The above $\LS_X$ and $\tMIC_X$ are small v-stacks in the sense of \cite{S3}.
\end{thm}



Similar to \cite{h1}, one can define Hitchin maps
\[H_X:\LS_X\to A_X,\]
\[H_X':\tMIC_X\to A_X\]
where $A_X$ is the Hitchin base of $X$, which is defined as the following $v$-sheaf on $\Perf_C$:
\[A_X: S\mapsto \bigoplus_{k\geq 0} H^0(\overline X_S, \Sym^k\Omega_{\overline X}(-1)).\]

More precisely, for any $\mathbb L \in \LS_X$, denote by  $\overline{\mathbb L}$ its reduction modulo $t$, which is an $\widehat{\mathcal O}_{\overline X}$-local system on $\overline X_{1,v}$. By \cite[Th. 4.2.1]{HX}, there exists a canonical Higgs field
\[\theta_{\overline{\mathbb L}}:\overline{\mathbb L}\to \overline{\mathbb L}\otimes_{\widehat{\mathcal O}_{\overline X}}\Omega^1_{\overline X}(-1)\]
satisfying some extra condition \cite[Th. 4.2.1 (a) and (b)]{HX}. Then we define $H_X(\mathbb{L})$ to be the coefficients of the characteristic polynomials of $\theta_{\overline{\mathbb L}}$. Similarly, for any $(D,\nabla_D)\in \tMIC_X$, denote by $(\overline D,\theta_{\overline D})$ its reduction modulo $t$, which is a Higgs field on $\overline X_{\et}$. Then we define $H_X'((D,\nabla))$ to be the coefficients of the characteristic polynomials of $\theta_{\overline D}$. A $\mathbb{B}_{\dR}^+$-local system $\mathbb{L}$ (resp. a integrable connection $(D,\nabla_D)$) is nilpotent if and only if $H_X(\mathbb{L}) = 0$ (resp. $H_X'((D,\nabla_D)) = 0$.)

Now we can also define stacks for our moduli problems.

\begin{cor}
The prestacks \begin{align*}
\LS^0_X:&\ S\mapsto \ls^0(\overline{X}_S),\\
\tMIC^0_X:&\ S\mapsto \tmic^0(X_S)
\end{align*}
where the superscript $0$ denote the subcategories of nilpotent objects, are small v-stacks.
\end{cor}
\begin{proof}
They are just the fibres at $0\in A_X$ of Hitchin maps.
\end{proof}

We can now state our main theorem.
\begin{thm}\label{thm:stacky RH}
There is an equivalence of stacks
\[RH^0:\LS^0_X\xrightarrow{\sim}\tMIC^0_X.\]

Namely, for any $S\in \Perf_C$, there is an equivalence of categories
\[RH^0_S:\ls^0(X_S)\xrightarrow{\sim}\tmic^0(X_S)\]
satisfying \[f^*_\et\circ RH^0_S\cong RH^0_{S'}\circ f^*_v\] 
for any map $f:S'\to S$ in $\Perf_C$.
\end{thm}

\begin{rmk}
    By smallness we may assume $S\in \Perf_C$ is topologically countably generated.
\end{rmk}

In the rest of this section, we refer to \cite[Section 2]{Yu} and remark that the setting in \emph{loc. cit.} fits over any perfectoid spaces over $C$.


  Now we focus on the construction for the desired equivalence of categories for $X_S$ as claimed in Theorem \ref{thm:stacky RH}. As we will always work with adic spaces over $\mathbb{B}_{\dR}^+(S)$, by abuse of notation, from now on, without other clarification, we always fix an $S\in \Perf_C$, put $B_{\dR}^+ = \mathbb{B}_{\dR}^+(S)$ and denote by $X$ a smooth adic space over $B_{\dR}^+ = \mathbb{B}_{\dR}^+(S)$ for short. In particular, the reduction $\overline X$ of $X$ modulo $t$ is a smooth adic space over $S$.

\subsection{Toric charts}\label{toric}
As for $\overline X_\et$, we have a base of $X_\et$ convenient for computations.

\begin{defn}
For $S\in \Perf_C$, denote the toric algebras over $B_\dR^+$ by
 \begin{align*}
  T^{d,?} = B_{\dR}^{+,?}\langle T_1^{\pm 1}, \ldots, T_d^{\pm 1}\rangle,
  T^{d,?} _m = B_{\dR}^{+,?}\langle T_1^{\pm 1/p^m}, \ldots, T_d^{\pm 1/p^m}\rangle, \text{ and }
  T^{d,?}_\infty =\dlim_m  T^{d,?} _m,   
 \end{align*} 
where $?\in\{\empty,+\}$ and for any $m\in\mathbb N\cup\{\infty\}$, put \begin{align*}
\mathbb T^d=\Spa \bigl(T^d,T^{d,+}\bigr) \text{ and }
\mathbb T^d_m=\Spa \bigl( T^d_m, T^{d,+}_m\bigr).
\end{align*}
For any smooth adic space $X$ over $B_\dR^+$, a toric chart of $X$ is a standard \'etale morphism $\psi: X\to \mathbb T^d$ of smooth adic spaces over $B^+_\dR$.
\end{defn}

\begin{prop}\label{etsite:more}
Let $X$ be an adic space over $B^+_{\dR}$. Then $X$ is smooth over $B^+_\dR$ if and only if there exists an open covering
\[ \{ U_i \subset X : i\in I \}\]
such that each $U_i$ admits a toric chart.
\end{prop}
\begin{proof}
See Proposition \ref{etsite} and \cite[Lemma 2.5]{LZ}.
\end{proof}

Let $X=\Spa(A,A^+)$ be a smooth adic space over $B_{\dR}^+$. Assume $X=\Spa(A,A^+)$ admits a toric chart $X\to \mathbb T^d$. We set some notations for local considerations.

Denote $X_n = X \times_{\mathbb T^d} \mathbb T^d_n$,
and $X_\infty = \ilim_n X_n$. Then $\overline X_{\infty}$ is the perfectoid space associated to $\varprojlim_m \overline X_m$ and $\overline X_{\infty}\to \overline X$ is a pro-\'etale Galois covering. Denote by $\Gamma$ its Galois group, then we have an isomorphism
 \[\Gamma\cong \mathbb Z_p(1)^d = \bigl\langle \gamma_1,\ldots,\gamma_d \bigr\rangle,\]
acting on $\overline X_{\infty}$ by $\gamma_i T_j^{1/p^n} = \zeta_{p^n}^{\delta_{ij}}T_j^{1/p^n},\quad 1\leq i,j\leq d, n\geq 0$. 

Write $X_n = \Spa \bigl( A_n, A_n^+ \bigr)$ and $A_\infty=\dlim_n A_n$.
As there is no canonical completion of $T^d_\infty$(see the discussion after Definition 2.12 in \cite{Yu}), we have to introduce the complete toric tower as in \cite{Yu}.

Let $\hat{\bar{\mathbb{T}}}_{\infty}^d$ be the perfectoid space associated to $\varprojlim \overline{\mathbb{T}}^d_m$. Then we have
  \[\hat{\bar{\mathbb{T}}}_{\infty}^d = \Spa(\hat{\bar T}_{\infty}^d,\hat{\bar T}_{\infty}^{d,+}) = \Spa(S\langle T_1^{\pm 1/p^\infty}, \ldots, T_d^{\pm 1/p^\infty}\rangle,S^+\langle T_1^{\pm 1/p^\infty}, \ldots, T_d^{\pm 1/p^\infty}\rangle).\]
Let $\widehat T_\infty^{d,?}=\mathbb B_\dR^{+,?}\bigl(\hat{\bar{\mathbb{T}}}_{\infty}^d\bigr)$. 
There is a map of $B_{\dR}^+$-algebras
\[T^d_\infty\to \widehat T_\infty^d\]
sending each $T_i^{1/p^n}$ to $[(T_i^{\flat})^{1/p^n}]$ for any $1\leq i\leq d$ and $n\geq 0$, which is an injection and identifies the former as a dense subalgebra of the latter.
Then $T^d_\infty\subset \widehat T_\infty^d$ is a dense subalgebra by map $T_i\mapsto \bigl[T_i^\flat \bigr]$, and 
\[\mathbb T_\infty^d=\ilim_n\mathbb T_n^d\sim\Spa \bigl(\widehat T_\infty^d,\widehat T_\infty^{d,+}\bigr)\]
in the sense of \cite[Definiton 2.4.1]{scholze2013modulipdivisiblegroups}

Put $\hat{\bar A}_\infty = \widehat{\mathcal O}_{\overline X}(\overline X_{\infty})$ and for any $?\in\{\emptyset,+\}$, put 
\[\widehat A^{?}_\infty=\mathbb B_\dR^{+,?}(\overline X_{\infty}).\]
By the \'etaleness of $X_n$ over $\overline{\mathbb{T}}^d_n$, there exists a morphism of $B_{\dR}^+$-algebras
\[A_{\infty}\to \widehat A_{\infty}\]
extending the map $T^d_\infty\to \widehat T_\infty^d$ above which is an injection and identifies the former as a dense sub-algebra of the latter. Moreover, by \cite[Proposition 2.9]{Yu}, we have an isomorphism
\[\widehat A_\infty\cong A\otimes_{T^d}\widehat T_\infty^d.\]
As $\overline X_{\infty}\to \overline X$ is a $\Gamma$-torsor, the $\widehat A_\infty$ carries a natural $\Gamma$-action, determined by that for any $1\leq i,j\leq d$ and any $n\geq 0$,
\[ \gamma_i \bigl[ \bigl( T_j^\flat \bigr)^{1/p^n} \bigr] = \bigl[ \epsilon^{\delta_{ij}/p^n} \bigr] \cdot \bigl[ \bigl( T_j^\flat \bigr)^{1/p^n} \bigr]. \]
Denote by $\Gamma^{p^n}$ the subgroup of $\Gamma$ corresponding to the Galois covering $\overline X_{\infty}\to \overline X_n$. By \cite[Proposition 2.25]{Yu}, we have 
\[A_n = (\widehat A_{\infty})^{\Gamma^{p^n}\text{-an}}\]
is the $B_{\dR}^+$-subalgebra consisting of elements on which $\Gamma^{p^n}$ acts analytically.
The $\mathbb B_\dR^+$-structure induces a natural $\Gamma$-action on $\widehat A_\infty$.
Under this action, $A_n=\widehat A_\infty^{p^n\Gamma-an}$. See \cite[Proposition 2.25]{Yu}.

\begin{prop}[\emph{\cite[Proposition 3.19]{Yu}}]\label{kl}
For $X$ above there is a canonical equivalence from the category of $\mathbb B_{\dR}^+$-local systems on $X_{1,v}$ of rank $r$ to the category of following data:
\[(M,\rho:\Gamma\times M\to M)\]
where $M$ is a projective module over $\widehat A_\infty$ of rank $r$ and $\rho$ is a continuous semilinear $\Gamma$-action.
\end{prop}
\begin{proof}
    See \cite[Theorem 3.5.8]{kl2}.
\end{proof}

With these constructions Yu develop the decompletion theory on toric charts.
\begin{prop}[\emph{\cite[Theorem 3.22]{Yu}}]\label{de}
The triple $\bigl(\dlim A_{n,\alpha} , \widehat A_{\infty,\alpha}, \Gamma\bigr)$ forms a strong decompletion system for any $\alpha\in\mathbb N\cup\{\infty\}$ in the sense of \cite[Appendix A]{diao}.
\end{prop}
For details, please consult \cite{Yu}.

\subsection{Period sheaves}

Here, we introduce the relative version of period sheaves with connections $(\mathcal O\mathbb{B}_{\dR}^+,d)$ and $(\mathcal O\mathbb{B}_{\dR},d)$ as in \cite{S2}.

Let $X$ be a smooth adic space $X$ over $B_\dR^+$. For any affinoid perfectoid $U=\Spa(R,R^+)\in X_{1,v}$, let $\Sigma_U$ be the set of all \'etale maps over $S$ \[i_{Y_1}: Y_1 = \Spa(A_1,A_1^+)\to X_1\] such that the natural map $U\to X$ factors through $i_{Y_1}$. Such $i_{Y_1}$ admits a unique lifting \[i_{Y}: Y = \Spa(A,A^+)\to X\] which is an \'etale map over $B_\dR^+$. There is a natural morphism
\[\theta_Y:A\widehat\otimes_{B_\dR^+}B_\dR^+(U)\to R\]
defined by $A\to A_1\to R$ and $\theta:B_\dR^+(U)\to R$.

\begin{defn}
Define the sheaf on $X_{1,v}$ \[ \OO\BB_{\dR(,X/S)}^+ =\left( U\mapsto\dlim_{\Sigma_U} \ilim_{j,\alpha} \left( \bigl( A_\alpha\widehat\otimes_{B_{\dR,\alpha}^+}B_{\dR,\alpha}^+(U)\bigr) /(\ker \theta_Y)^j\right)\right)^\# \]
with a map $\theta: \mathcal O \mathbb B_\dR^+ \to \widehat{\mathcal O}_{\overline X}$ induced by $\theta_Y$'s.

It admits the filtration $\Fil^j \obdr^+=(\ker \theta)^j$.

It admits the connection
\[d: \mathcal O \mathbb B_\dR ^+ \to \mathcal O \mathbb B_\dR^+ \otimes_{\nu^{-1} \mathcal O_X}\nu^{-1} \Omega_X\] induced by the universal one
$A\to A\otimes \Omega_X$, where we write $\Omega^1_X=\Omega^1_{X/\mathbb B_\dR^+(S)}$ the sheaf of differentials in the sense of \cite{hub}.

Moreover, let
$\mathcal O \mathbb B'_\dR = \mathcal O \mathbb B_\dR^+[t^{-1}]$
which is equipped with a filtration \[\Fil^j \mathcal O \mathbb B'_\dR = \sum_{i\in \mathbb Z} t^i \Fil^{j-i} \mathcal O \mathbb B_\dR^+.\]

Define $\mathcal O \mathbb B_\dR$ to be the completion of $\mathcal O \mathbb B'_\dR$ along this filtration. It admits the connection 
\[d: \mathcal O \mathbb B_\dR \to \mathcal O \mathbb B_\dR\otimes_{\nu^{-1} \mathcal O_X}\nu^{-1} \Omega_X\]
satisfying Griffiths transversality, which is a $t$-connection on $\Fil^0 \mathcal O \mathbb B_\dR$
\[d: \Fil^0 \mathcal O \mathbb B_\dR \to \Fil^{-1} \mathcal O\mathbb B_\dR \otimes_{\nu^{-1}\mathcal O_X}\nu^{-1} \Omega_X. \]

\end{defn}

\begin{prop}[\emph{\cite[Proposition 2.34]{Yu}}]\label{obdr+}
Assume $X$ admits a toric chart. Then the natural map
\[\iota:\mathbb B^+_{\dR|X_\infty} [[U_1,\ldots,U_d]] \to
\mathcal O \mathbb B^+_{\dR|X_\infty}\]
sending $U_i$ to $T_i - \bigl[ T_i^\flat \bigr]$ is an isomorphism of sheaves of $\mathbb B_\dR^+$-algebra compatible with filtration, where the left side is equipped with the $(t,U_1,\dots,U_d)$-adic filtration.
Via this isomorphism,\[
\iota^{-1}d = \sum_{i=1}^d\frac{\partial}{\partial U_i}\otimes d T_i.
\]
\end{prop}

\begin{cor}\label{obdr}
The natural map
\[\iota:\mathbb B^+_{\dR|X_\infty}\langle V_1,\ldots,V_d\rangle \to
\mathcal O \mathbb B^{[0,+\infty]}_{\dR|X_\infty}\]
sending $V_i$ to $\frac 1 t \log \frac{[ T_i^\flat]}{T_i}$ is an isomorphism compatible with filtration. Moreover,\[
\iota^{-1}d = -\sum_{i=1}^d\frac{\partial}{\partial V_i}\otimes \frac{d \log T_i}{t}.
\]
\end{cor}

\begin{cor}\label{Poincare}
The de Rham sequence $0\to\mathbb B_\dR^+\to DR(\mathcal O\mathbb B_\dR^+,d)$ is exact. Similarly, the de Rham sequence $0\to\mathbb B_\dR^+\to DR(\mathcal O\mathbb B_\dR^{[0,\infty]},d)$ is exact.
\end{cor}
\begin{proof}
Check on toric charts.
\end{proof}

The above construction is compatible with base change of $S\in\Perf_C$ in the following sense.

\begin{prop}\label{ext}
For any morphism $f:S'\to S\in \Perf_C$, the natural morphism
\[
f^{-1}\OO\BB_{\dR,X/S}^+\widehat\otimes_{\mathbb B_\dR^+(S)}\mathbb B_\dR^+(S')\to\OO\BB_{\dR,X_{S'}/S'}^+
\]
is an isomorphism compatible with filtration and connection.
\end{prop}

\begin{proof}
Assume $X$ admits a toric chart by locality. We then deduce by 
Proposition \ref{obdr+} and \cite[Theorem 4.6]{Yu}.
\end{proof}
\begin{cor}
There is a natural isomorphism compatible with connection
\[
f^{-1}\OO\BB_{\dR,X/S}^{[0,+\infty]}\widehat\otimes_{\mathbb B_\dR^+(S)}\mathbb B_\dR^+(S')\to\OO\BB_{\dR,X_{S'}/S'}^{[0,+\infty]}.
\]
\end{cor}

\begin{prop}\label{coh}
Let $\mathbb L$ be a $\mathbb B_\dR^+$-local system on $X_{1,v}$. For any 
$-\infty\leq a\leq b\leq\infty$ and any affinoid perfectoid $U\in X_{1,v}$,
\[H^i(U, \mathbb L \otimes_{\mathbb B_\dR^+} \obdr^{[a,b]})=0,\quad i>0.\]
\end{prop}

\begin{proof}
By the same induction argument as in \cite[Lemma 2.4]{LZ} and \cite[Lemma 3.18]{S2}, we are reduced to the case $a=b$. By shift, we may further assume $a = b = 0$.

Taking the $\gr^1$-piece of the de Rham sequence in Corollary \ref{Poincare}, we get a short exact sequence
\[0\to \widehat{\mathcal{O}}_{\overline X}(1)\to \gr^1\mathcal{O}\mathbb{B}_{\dR}^+\to \widehat{\mathcal{O}}_{\overline X}\otimes\Omega^1_{\overline X}\to 0\]
and thus a short exact sequence
\[0\to \widehat{\mathcal{O}}_{\overline X}\to \mathcal{F}\to \widehat{\mathcal{O}}_{\overline X}\otimes\Omega^1_{\overline X}(-1)\to 0,\]
where $\mathcal F = \left(\gr^1\mathcal{O}\mathbb{B}_{\dR}^+\right)(-1)$ denotes the Faltings's extension.
As in \cite[Remark 2.1]{LZ}, we have an isomorphism
\[\gr^0 \mathcal O \mathbb B_\dR\cong \dlim_n \Sym^n_{\widehat{\OO}_{\overline X}} \mathcal F.\]
As the right-hand-side above is a direct limit of $\widehat{\OO}_{\overline X}$-bundles, by the quasi-compactness of $U$,
\[
H^i(U,\mathbb L\otimes_{\mathbb B_\dR^+} \gr^0\obdr)= \dlim_nH^i(U,\mathbb L_1\otimes\Sym^n_{\widehat{\OO}_{\overline X}} \mathcal F).
\]
By \cite[Proposition 8.8]{S3} and \cite[Theorem 3.5.8]{kl2}, for any $i\geq 1$, we have 
\[H^i(U,\mathbb L_1\otimes\Sym^n_{\widehat{\OO}_{\overline X}} \mathcal F)= 0\]
which implies that 
\[H^i(U,\mathbb L\otimes_{\mathbb B_\dR^+} \gr^0\obdr) = 0\]
as desired.
\end{proof}

\begin{cor}[\emph{\cite[Theorem 2.38]{Yu}}]\label{pr}
We have

(1) $R\nu_*\obdr^{[a,b]}=\OO_{X}[\frac{1}{t}]^{[a,b]}$, where we define $\Fil^j(\OO_{X}[\frac{1}{t}])=t^j\OO_{X}$.

(2) For any vector bundle $\mathcal E$ on $X$,
\[R\nu_*\bigl(\nu^{-1}\mathcal{E}\otimes_{\nu^{-1}\OO_{X}} \obdr^{[0,+\infty]}\bigr)=\mathcal{E}.\]
\end{cor}

\section{Proof of Main Theorem}

We now give the following equivalence of categories:
\begin{thm}\label{thm:categorical RH}
Let $X$ be a smooth adic space over $B_\dR^+$. Then the following functors
\begin{align*}
\ls(X)^0& \longleftrightarrow \tmic (X)^0 \\ \mathbb L& \mapsto \nu_* \bigl(\mathbb L\otimes_\bdr\Fil^0\obdr, d\otimes1 \bigr) \\
\bigl(\nu^{-1}\mathcal E\otimes_{\nu^{-1} \OO_X} \Fil^0 \obdr\bigr)^{\nabla\otimes d=0}&\mapsfrom (\mathcal E,\nabla)
\end{align*}
induce an equivalence of categories such that for any $\mathbb{B}_{\dR}^+$-local system $\mathbb L$ with corresponding integrable connection $(\mathcal E,\nabla)$ via the above equivalence, there is a quasi-isomorphism
\[
R\nu_*\mathbb L=DR\bigl(\mathcal E,\nabla\bigr).
\]
\end{thm}
With this construction, we obtain Theorem \ref{main}. 

\begin{proof}[\textbf{Proof of Theorem \ref{main}}]
The rest is to check $f^*_\et\circ RH^0_S\cong RH^0_{S'}\circ f^*_v$ for any map $S'\xrightarrow{f} S$ in $\Perf_C$. This natural morphism is given by \begin{align*}
f^*_\et\nu_* \bigl(\mathbb L\otimes_{\mathbb B_\dR^+} \mathcal O \mathbb B_{\dR,X_S/S}^{[0,+\infty]} \bigr) \xrightarrow{\adj}& f^*_\et\nu_* f_*f^*_v\bigl(\mathbb L\otimes_{\mathbb B_\dR^+} \mathcal O \mathbb B_{\dR,X_S/S}^{[0,+\infty]} )\\=&f^*_\et f_*\nu_*\bigl (f^*_v \mathbb L\otimes_{\mathbb B_\dR^+}\mathcal O \mathbb B_{\dR,X_{S'}/S'}^{[0,+\infty]} )\\\xrightarrow{\adj}&\nu_*\bigl(f^*_v \mathbb L\otimes_{\mathbb B_\dR^+}\mathcal O \mathbb B_{\dR,X_{S'}/S'}^{[0,+\infty]} ) 
\end{align*}
Then by locality we assume $X$ admits a toric chart. And by \cite{kl2} we only need to check their global sections. Then our theorem follows from Proposition \ref{ext} and \cite[Proposition 2.9]{Yu}. It holds
\[RH^0_{S'}(f^*_v\mathbb L)(X_{S})=RH^0_S(\mathbb L(X_{S'}))\otimes_{B_\dR^+(S)}B_\dR^+(S')
\]
which means $\adj(X_S)$'s are isomorphisms. We get the desired theorem.
\end{proof}
\subsection{Local construction}

In this subsection assume $X$ admits a toric chart and take notations as in Subsection \ref{toric}.

\begin{lem}\label{analytic}
Let $\overline{\mathbb L}$ be an $\widehat{\mathcal O}_{\overline X}$-bundle on $X_{1,v}$.
\begin{enumerate}[label=(\alph*)]
\item There exists a finite projective $\overline A_\infty$-module $\overline M_\infty$ with a semilinear $\Gamma$-action, unique up to unique isomorphism, such that for any affinoid perfectoid $U$ over $X_{\infty}$, there is a natural isomorphism of modules with semilinear $\Gamma$-actions
\[ \overline{\mathbb L} \bigl( U \bigr) = \overline M_\infty \otimes_{\overline A_\infty} \widehat{\mathcal O}_{\overline X} \bigl( U \bigr). \]
Moreover, there exists $N \gg 1$ and some projective module $\overline M_N$ over $\overline A_N$ equipped with a semilinear $\Gamma$-action, such that
\[ \overline M_\infty \cong \overline M_N \otimes_{\overline A_N} \overline A_\infty. \]

\item With the construction of $\overline M_\infty$ above, $\overline{\mathbb L}$ is nilpotent if and only if the action of $\Gamma$ on $\overline M_\infty$ is quasi-unipotent, if and only if the action of $\Gamma$ on $\overline M_N$ is quasi-unipotent.
\end{enumerate}
\end{lem}

\begin{proof}
Denote $\hat{\bar M}_\infty=\overline{\mathbb L} (\hat{\bar{U}} _\infty)$. By Proposition \ref{kl} this is a finite projective $\hat{\bar A}_\infty$-module with a semilinear $\Gamma$-action, such that for any affinoid perfectoid $U$ over $X_{\infty}$,
\[ \overline{\mathbb L} \bigl( U \bigr) = \hat{\bar M}_\infty \otimes_{\hat{\bar A}_\infty} \widehat{\mathcal O}_{\overline X} \bigl( U \bigr).\]

By Proposition \ref{de} there exists a finite projective $\overline A_\infty$-module $\overline M_\infty$ with a semilinear $\Gamma$-action such that $\hat{\bar M}_\infty \cong \overline M_\infty \otimes_{\overline A_\infty} \hat{\bar A}_\infty.$ Moreover, there exists some $N\gg 1$ and some projective module $\overline M_N$ over $\overline A_N$ such that $\overline M_\infty \cong \overline M_N \otimes_{\overline A_N} \overline A_\infty.$

Let $\bigl\{ m_j \bigr\}$ be a set of generators for $\overline M_N.$ There exists $N' \ge N$ such that
\[ \gamma_i \bigl( m_j \bigr) \in \overline M_N \otimes_{\overline A_N} \overline A_{N'} \]
for all $i,j.$ Replace $N$ by $N'$ and $\overline M_N$ by $\overline M_N \otimes_{\overline A_N} \overline A_{N'},$ we may assume that $\overline M_N$ is $\Gamma$-invariant. This completes the proof of part (a).

As for part (b), recall the definition of the geometric Sen operator
\[ \vartheta \bigl( x \bigr) = \sum_i \vartheta_i \bigl( x \bigr) \otimes \frac{\d \log T_i}{t}, \]
where
\[ \vartheta_i \bigl( x \bigr) = \lim_{u \in \mathbb Z_p \to 0} \frac{\gamma_i^u \bigl( x \bigr) - x} u \]
is $\hat{\bar A}_\infty$-linear, see for example, \cite[Remark 4.12]{liu2022rhamprismaticcrystalsmathcalok}. 
Now, $\overline A_N\subset \hat{\bar A}_\infty$ is equipped with the spectral norm with unit ball $\overline A_N ^+$. As $\overline M_N$ is the direct summand of some $\overline A_N^{\oplus k},$ we get a norm on $\overline M_N,$ hence also on
$\End_{\overline A_N} \bigl( \overline M_N \bigr)$.

Consider the restriction
\[ \varphi_i = \bigl. \gamma_i^N \bigr|_{\overline M_N} \in \Aut_{\overline A_N} \bigl(\overline M_N \bigr). \]
It follows that $\varphi_i^{p^m} \to I, m \to \infty.$ Take $m \in \mathbb Z_{\geq 1}$ large enough such that for all $i$ we have
\[ \bigl\| \varphi_i^{p^m} - I \bigr\| < p^{-1/( p-1 )}. \]
It follows that $\bigl. \vartheta_i \bigr|_{\overline M_N} = \log \varphi_i^{p^m}$ and
\[ \varphi_i^{p^m} - I = \exp \bigl( \bigl. \vartheta_i \bigr|_{\overline M_N} \bigr) - I. \]
It follows that the $\Gamma$-action on $\overline M_N$ is quasi-unipotent if and only if the restriction of $\vartheta_i$ on $\overline M_N$ is nilpotent. This proves part (b).
\end{proof}

\begin{cor}\label{analytic:nilpotent}
Let $\overline{\mathbb L}$ be a nilpotent $\widehat{\mathcal O}_{\overline X}$-bundle on $X_{1,v}$. Then for any affinoid perfectoid $U$ over $X_{\infty}$, there exists a finite projective $\overline A$-module $\overline M$ equipped with a unipotent linear $\Gamma$-action, such that
\[ \overline{\mathbb L} \bigl( U \bigr) = \overline M \otimes_{\overline A} \widehat{\mathcal O}_{\overline X} \bigl( U \bigr). \]
Moreover,
\[H^i\bigl(\Gamma, \overline{\mathbb L}(X_\infty)\bigr)\cong H^i\bigl(\Gamma, \overline M).\]
\end{cor}

\begin{proof}[Proof]
Consider the module $\overline M_\infty$ as in Lemma \ref{analytic}.

For each $N \ge 0,$ take
\[ \overline M_N = \bigl\{ x \in \overline M_\infty : \forall i, \exists m, \bigl( \gamma_i^{1/p^N} - 1 \bigr)^m x =0 \bigr\}. \]
Over $S / C$, we may consider the decomposition of $\overline M_N$ into generalised eigenspaces. Write
\[ \overline M_N = \bigoplus_{0 \le k_1,\ldots,k_d < p^N} \overline M_N^{( k_1,\ldots,k_d )}\]
where $\gamma_i^{1/p^N} - \zeta_{p^N}^{k_i}$ acts nilpotently on $\overline M_N^{( k_1,\ldots,k_d )}$. Note that
\[ \overline M_N^{( k_1,\ldots,k_d )} = T_1^{k_1/p^N} \cdots T_d^{k_d/p^N} \overline M_0.\]
Then the natural map $M_0 \otimes_{\overline A} \overline A_N \to M_N$ is an isomorphism. Taking the colimit we see that
\[ \overline M_0 \otimes_{\overline A} \overline A_\infty \to \overline M_\infty \]
is an isomorphism. Now, just take $\overline M = \overline M_0.$
\end{proof}

\subsection{The functor RH}

\begin{lem}\label{pushforward}
Assume $X$ admits a toric chart. Let $\overline{\mathbb L}$ be a nilpotent $\widehat{\mathcal O}_{\overline X}$-local system on $X_{1,v}$ and take $\overline M$ as in Corollary \ref{analytic:nilpotent}.

There is a natural isomorphism
\[ R\Gamma \bigl( X, \overline{\mathbb L} \otimes_{\widehat{\mathcal O}_{\overline X}} \gr^0 \mathcal O \mathbb B_\dR \bigr) \to \overline M. \]
\end{lem}

\begin{proof}
The proof goes essentially the same as in \cite[Lemma 2.9]{LZ}.

Note that the $j$-fold fibre product $X_\infty^{j/X}\simeq X_\infty\times\Gamma^{j-1}$ in $X_\proet$. Hence
\[H^i(X_\infty^{j/X},\cdot)\cong\Hom_{\mathrm{cont}}\bigl(\Gamma^{j-1}, H^i\bigl(X_\infty,\cdot\bigr)\bigr).\]

Using Proposition \ref{coh} and computing the Cartan-Leray spectral sequence for $X_\infty\to X$, we get
\begin{align*}
H^i(X, \mathbb L\otimes_{\widehat{\mathcal O}_{\overline X}} \gr^0 \obdr)&\cong H^i\bigl(\Gamma, (\mathbb L\otimes_{\widehat{\mathcal O}_{\overline X}} \gr^0 \obdr )(X_\infty)\bigr)\\{}&\cong H^i\bigl(\Gamma, M\bigl[V_1,\dots,V_d]\bigr).
\end{align*}

Compute this term by \cite[Lemma 2.10]{LZ} inductively by the fact that $\Gamma$ acts on $\overline M$ unipotently, we deduce the result.
\end{proof}

\begin{thm}\label{LS:MIC}
Let $\mathbb L$ be a nilpotent $\mathbb B_\dR^+$-local system on $X_{1,v}$. Denote
\[ \mathcal E = R\nu_* \bigl( \mathbb L\otimes_\bdr\Fil^0 \obdr  \bigr), \]
with a natural integrable $t$-connection
\[ \nabla: \mathcal E \to \mathcal E \otimes_{\mathcal O_X} \Omega_X\{-1\} \]
induced by that on $\Fil^0 \mathcal O \mathbb B_\dR$. We then have:

\begin{enumerate}[label=(\alph*)]
\item $\mathcal E \cong \nu_* \bigl( \mathbb L\otimes_\bdr\Fil^0 \obdr  \bigr)$ is an $\mathcal O_X$-bundle on $X_\et$ with a nilpotent $t$-connection $\nabla$. \\

\item The natural map
\[ \nu^{-1} \mathcal E\otimes_{\nu^{-1} \mathcal O_X} \Fil^0 \obdr  \to \mathbb L\otimes_\bdr\Fil^0 \obdr  \]
is an isomorphism between modules with connections.\\

\item The natural map
\[ \mathbb L \to \bigl( \mathbb L\otimes_\bdr\Fil^0 \obdr  \bigr)^{\nabla=0} \cong \bigl( \nu^{-1} \mathcal E\otimes_{\nu^{-1} \mathcal O_X} \Fil^0 \obdr  \bigr)^{\nabla=0} \]
is an isomorphism.
\end{enumerate}

Also, $t$ is a nonzero-divisor on $\mathcal O_X$, $\mathcal O_X$ has vanishing higher cohomology on $X$, and
\[ \mathcal O_{X_\alpha} = \mathcal O_X / t^\alpha. \]
\end{thm}

\begin{proof}
Take
$\mathcal E_\alpha = R\nu_* \bigl( \mathbb L\otimes_\bdr\Fil^0 \obdr  / t^\alpha \bigr)$.

Using Lemma \ref{pushforward}, taking sheafification shows that the coherent sheaf $\mathcal E_1$ on $\overline X_\et$ defined by the datum of finite projective $\overline A$-module $\overline M$ is a vector bundle. By induction, $\mathcal E_\alpha$ is a coherent sheaf on $X_{\alpha,\et}$ for all $\alpha$.

By shrinking $X$ we may assume that $\overline M$ is free over $\overline A.$ Let $e_1,\ldots,e_r$ be a basis for $\overline M$. Since $t$ is a nonzero-divisor in $\Fil^0 \mathcal O \mathbb B_\dR$, there are exact sequences for $\alpha,\beta>0$ 
\[0\to M_\alpha\xrightarrow{t^\beta}M_{\alpha+\beta}\to M_\beta\to0\]
In particular $M_{\alpha+1} \to M_\alpha$ is surjective, so $\overline M = M_\alpha / t.$ Using these surjections, we may lift $e_i$ so that $e_i \in \ilim_\alpha M_\alpha.$ Now, by Nakayama's Lemma, $e_i$'s generate $M_\alpha$ over $A_\alpha.$ By induction, we see that the maps
\[ \bigl( e_1,\ldots,e_r \bigr): A_\alpha^{\oplus r} \to M_\alpha \]defined by
\[(f_1, \dots, f_r) \mapsto \sum _{i=1} ^r e_i f_i\]
are isomorphisms 
Taking the inverse limit, by \cite[Lemma 3.18]{S2}, we get an isomorphism
\[ \mathcal O_X^{\oplus r} \to \mathcal E, \]
proving (a). The nilpotency follows from Corollary \ref{analytic:nilpotent} as in \cite[Lemma 2.11]{LZ}.

In particular, we may take $\mathbb L = \mathbb B_\dR^+$ the trivial local system. For $\alpha,\beta>0,$ the exact triangle
\[0\to \Fil^0 \mathcal O \mathbb B_\dR / t^\alpha \xrightarrow{t^\beta}
\Fil^0 \mathcal O \mathbb B_\dR / t^{\alpha+\beta}\to\Fil^0 \mathcal O \mathbb B_\dR / t^\beta\to 0\]
will induce the exact sequence
\[0\to\mathcal O_{X_\alpha} \xrightarrow{t^\beta}\mathcal O_{X_{\alpha+\beta}} \to\mathcal O_{X_\beta}\to 0.\]
Taking the global sections, we get an exact sequence
\[0\to A_\alpha\xrightarrow{t^\beta}A_{\alpha+\beta}\to A_\beta\to0.\]
Taking the inverse limit with respect to $\alpha$, we see that there is an exact sequence
\[0\to A \xrightarrow{t^\beta}A\to A/t^\beta\to0\]
where $A = \ilim A_\alpha.$ This shows that $t$ is a nonzero-divisor on $\mathcal O_X$ and $\mathcal O_X / t^\beta = \mathcal O_{X_\beta}.$ The vanishing cohomology part follows from \cite[Lemma 3.18]{S2}.

As for (b), using a limiting argument, we only need to show that
\[ \nu^{-1} \mathcal E\otimes_{\nu^{-1} \mathcal O_X} \Fil^0 \OO\BB_{\dR,\alpha} \to \mathbb L\otimes_\bdr\Fil^0 \obdr  / t^\alpha \]
are isomorphisms. Using d\'evissage, we only need to prove this for $\alpha = 1.$ This is done in Theorem \ref{analytic}.

At last, for (c), the map
\[ \mathbb L \to \bigl( \mathbb L\otimes_\bdr\Fil^0 \obdr   \bigr)^{\nabla=0} \]
is an isomorphism by Corollary \ref{Poincare}.
\end{proof}

\subsection{The inverse functor}

\begin{lem}\label{nilpotent:connection}
Let $R$ be a ring over $\mathbb Q$ and $M$ be an $R \bigl[ x_1,\ldots,x_n \bigr]$-module with a nilpotent integrable $R$-linear connection
\[ \nabla: M \to \bigoplus_{i=1}^n M \d x_i. \]
Then the natural map
\[ M^{\nabla=0} \otimes_R R \bigl[ x_1,\ldots,x_n \bigr] \to M \]
is an isomorphism.
\end{lem}

\begin{proof}
Write $\nabla = \bigl( \nabla_i \bigr)_{i=1,\dots, n}$ for the $n$ components of the connection.

We have
\[ M^{\nabla=0} \otimes_R R \bigl[ x_1,\ldots,x_n \bigr] = \bigoplus_{i_1 ,\dots, i_n \geq 0} M^{\nabla=0} x_1^{i_1} \cdots x_n^{i_n}. \]
The image of each summand has pairwise distinct eigenvalues (over $\mathbb Q$) for the commuting operators $\bigl( t_i\nabla_i \bigr)$, so the map
\[ M^{\nabla=0} \otimes_R R \bigl[ x_1,\ldots,x_n \bigr] \to M \]
is injective.

Now, denote the image by $M'$, and we will show that $M'=M.$ For $m \in M,$ denote $k_i$ the minimal non-negative integer with
\[ \nabla_i^{k_i+1} \bigl( m \bigr) =0. \]
By induction on $k_1+\cdots+k_n,$ we show that $m \in M'.$

Without loss of generality, assume $k_1>0$.
Take
\[ m_1 = \nabla_1 \bigl( m \bigr), m_2 = m - \frac{x_1}{k_1} m_1. \]
We have $\nabla_i^{k_i+1} \bigl( m_j \bigr) =0$ for $i \ge 2$ and $j=1,2.$ We have $\nabla_1^{k_1} \bigl( m_1 \bigr) =0$ and
\[ \nabla_1^{k_1} \bigl( m_2 \bigr) = \nabla_1^{k_1} \bigl( m \bigr) - \nabla_1^{k_1-1} \bigl( m_1 \bigr) - x_1\nabla_1^{k_1} \bigl( m_1 \bigr) = 0. \]
By inductive hypothesis, we get $m_1,m_2 \in M',$ hence
\[ m = \frac{x_1}{k_1} m_1 + m_2 \in M'. \qedhere \]
\end{proof}

\begin{lem}\label{quasi:unipotent}
Let $\Gamma = \langle \gamma \rangle \cong \mathbb Z_p$. Let $M$ be a $\mathbb Q_p$-module equipped with a linear $\Gamma$-action. Equip $M \bigl[ x \bigr]$ with the action
\[ \gamma(x) = x+1, \]
then $M$ is $\Gamma$-quasi-unipotent if and only if $M \bigl[ x \bigr]$ is.
\end{lem}

\begin{proof}
This is clear by \cite[Lemma 2.10]{LZ}.
\end{proof}

\begin{thm}\label{MIC:LS}
Let $\mathcal E$ be a vector bundle with a nilpotent integrable $t$-connection. Then
\[ \mathbb L := \bigl(  \nu^{-1} \mathcal E\otimes_{\nu^{-1} \mathcal O_X} \Fil^0 \obdr  \bigr)^{\nabla=0} \]
is a nilpotent $\mathbb B_\dR^+$-local system on $X_{1,v}$, such that

(1) The natural map
\[ \mathbb L\otimes_\bdr\Fil^0 \obdr  \to  \nu^{-1} \mathcal E\otimes_{\nu^{-1} \mathcal O_X} \Fil^0 \obdr  \]
is an isomorphism.\\

(2) The natural map
\[ \mathcal E \to R \nu_* \bigl(  \nu^{-1} \mathcal E\otimes_{\nu^{-1} \mathcal O_X} \Fil^0 \obdr  \bigr) \cong R\nu_* \bigl( \mathbb L\otimes_\bdr\Fil^0 \obdr  \bigr) \]
is a quasi-isomorphism.

\end{thm}

\begin{proof}
Assume $X$ admits a toric chart by locality. As for (a), take $\mathcal E_\alpha = \mathcal E / t^\alpha$ and
\[ \mathbb L_\alpha = \bigl(  \nu^{-1} \mathcal E\otimes_{\nu^{-1} \mathcal O_X} \Fil^0 \OO\BB_{\dR,\alpha} \bigr)^{\nabla=0}. \]
Note that $\mathcal E_\alpha$ is an $\mathcal O_{X_\alpha}$-bundle by the last sentence of Theorem \ref{LS:MIC}.

For $U\in X_{\infty,v}$, we have isomorphisms
\[ \mathbb B_{\dR,\alpha}^+ \bigl( U \bigr) \bigl[ V_1,\ldots,V_d \bigr] \to \Fil^0 \mathcal O\mathbb B_\dR / t^\alpha \]
and
\[ \bigoplus_{i=1}^d \mathbb B_{\dR,\alpha}^+ \bigl( U \bigr) \bigl[ V_1,\ldots,V_d \bigr] \d V_i \to \Fil^0 \mathcal O \mathbb B_\dR / t^\alpha \otimes_{\nu^{-1} \mathcal O_X} t^{-1} \nu^{-1} \Omega_X \]
given by $V_i \mapsto t^{-1} \log \bigl( \bigl[ T_i^\flat \bigr]/ T_i \bigr)$ and $\d V_i \mapsto -t^{-1} \d \log T_i.$ By Lemma \ref{nilpotent:connection}, we see that there is an isomorphism
\[ \mathbb L\otimes_\bdr\Fil^0 \OO\BB_{\dR,\alpha} \to  \nu^{-1} \mathcal E\otimes_{\nu^{-1} \mathcal O_X} \Fil^0 \OO\BB_{\dR,\alpha}. \]
In particular, $\mathbb L_\alpha \bigl( U \bigr)$ is finite projective over $\mathbb B_{\dR,\alpha}^+ \bigl( U \bigr)$ by faithfully flat descent, so $\mathbb L_\alpha$ is indeed a $\mathbb B_{\dR,\alpha}^+$-local system for each $\alpha.$ Now, take the inverse limit, we get an isomorphism
\[ \mathbb L\otimes_\bdr\Fil^0 \obdr  \to  \nu^{-1} \mathcal E\otimes_{\nu^{-1} \mathcal O_X} \Fil^0 \obdr . \]

It remains to show that $\mathbb L_1 =: \overline{\mathbb L}$ is nilpotent. As above, we have an isomorphism
\[ \mathbb L_1\otimes_\bdr\Fil^0 \obdr\cong  \nu^{-1} \mathcal E\otimes_{\nu^{-1} \mathcal O_X} \Fil^0 \obdr/t \]
Let $U$ be pro-\'etale over $X_\infty$ and then this isomorphism can be written as
\[ \overline{\mathbb L} \bigl( U \bigr) \bigl[ V_1,\ldots,V_d \bigr] \cong \mathcal E_1 \bigl( U \bigr) \bigl[ V_1,\ldots,V_d \bigr]. \]
Note that this isomorphism is $\Gamma$-equivariant. Take
\[ \overline E_\infty = \dlim \mathcal E_1 \bigl( X_i \bigr), \]
then the action of $\Gamma$ on $\overline E_\infty$ is evidently quasi-unipotent, and we have a natural isomorphism for all $U$ as below
\[ \mathcal E_1 \bigl( U \bigr) \cong \overline E_\infty \otimes_{\overline A_\infty} \widehat{\mathcal O}_{\overline X} \bigl( U \bigr). \]
Now, the triple
\[ \bigl( \dlim \overline A_i \bigl[ V_1,\ldots,V_d \bigr], \hat{\bar A}_\infty \bigl[ V_1,\ldots,V_d \bigr], \Gamma \bigr) \]
is a strongly decompletion system, so there exists a unique isomorphism (where $\overline M_\infty$ is as in Lemma \ref{analytic}) of finite projective $\overline A_\infty$-modules equipped with semilinear $\Gamma$-actions
\[ \overline M_\infty \bigl[ V_1,\ldots,V_d \bigr] \cong \overline E_\infty \bigl[ V_1,\ldots, V_d \bigr]. \]
We conclude by an inductive application of Lemma \ref{quasi:unipotent}.

And (b) is clear by Corollary \ref{pr}.
\end{proof}

\begin{proof}[\textbf{Proof of Theorem \ref{thm:categorical RH}}]
Combining Theorem \ref{LS:MIC} and Theorem \ref{MIC:LS} we complete the proof. The last assertion follows from Corollary \ref{Poincare}.
\end{proof}

\bibliographystyle{alpha}
\bibliography{ref}

\newcommand{\etalchar}[1]{$^{#1}$}
\begin{thebibliography}{LMN{\etalchar{+}}24}

\bibitem[AHB23]{h}
Johannes Anschütz, Ben Heuer, and Arthur-César~Le Bras.
\newblock The small $p$-adic simpson correspondence in terms of moduli spaces, 2023.

\bibitem[DLLZ18]{diao}
Hansheng Diao, Kai-Wen Lan, Ruochuan Liu, and Xinwen Zhu.
\newblock Logarithmic {R}iemann-{H}ilbert correspondences for rigid varieties.
\newblock {\em arXiv preprint arXiv:1803.05786}, 2018.

\bibitem[Heu24]{h1}
Ben Heuer.
\newblock Moduli spaces in $p$-adic non-abelian {H}odge theory, 2024.

\bibitem[Hub96]{hub}
Roland Huber.
\newblock {\'E}tale cohomology of rigid analytic varieties and adic spaces, 1996.

\bibitem[HX24]{HX}
Ben Heuer and Daxin Xu.
\newblock Geometric {S}en theory over rigid analytic spaces, 2024.

\bibitem[KL19]{kl2}
Kiran~S. Kedlaya and Ruochuan Liu.
\newblock Relative $p$-adic {H}odge theory, ii: Imperfect period rings, 2019.

\bibitem[Liu22]{liu2022rhamprismaticcrystalsmathcalok}
Zeyu Liu.
\newblock De {R}ham prismatic crystals over $\mathcal{O}_k$, 2022.

\bibitem[LMN{\etalchar{+}}24]{lmnqw}
Yudong Liu, Chenglong Ma, Xiecheng Nie, Xiaoyu Qu, and Yupeng Wang.
\newblock A stacky {$p$-}adic {R}iemann-{H}ilbert correspondence on {H}itchin-small locus.
\newblock {\em In preparation}, 2024.

\bibitem[LZ17]{LZ}
Ruochuan Liu and Xinwen Zhu.
\newblock Rigidity and a {R}iemann-{H}ilbert correspondence for {$p$}-adic local systems.
\newblock {\em Invent. Math.}, 207(1):291--343, 2017.

\bibitem[Sch13]{S2}
Peter Scholze.
\newblock {$p$}-adic {H}odge theory for rigid-analytic varieties.
\newblock {\em Forum Math. Pi}, 1:e1, 77, 2013.

\bibitem[Sch22]{S3}
Peter Scholze.
\newblock {\'E}tale cohomology of diamonds, 2022.

\bibitem[SW13]{scholze2013modulipdivisiblegroups}
Peter Scholze and Jared Weinstein.
\newblock Moduli of $p$-divisible groups, 2013.

\bibitem[Yu24]{Yu}
Jiahong Yu.
\newblock On smooth adic spaces over {$\mathbb{B}_{\text{dR}}^+$} and sheafified $p$-adic {R}iemann-{H}ilbert correspondence, 2024.

\end{thebibliography}
\end{document}